\documentclass[11pt]{amsart}
\usepackage{amsmath, amsthm,amsxtra}
\usepackage[english]{babel}
\usepackage[T1]{fontenc}
\usepackage[latin1]{inputenc}
\usepackage{amssymb}
\usepackage{amscd}
\usepackage{latexsym}
\usepackage[bookmarksnumbered,plainpages,hypertex]{hyperref}
\usepackage{graphicx}
\usepackage{mathrsfs} 
\usepackage[all,cmtip]{xy}

\newtheoremstyle{theorem}
  {12pt}          
  {12pt}  
  {\sl}  
  {\parindent}     
  {\bf}  
  {. }    
  { }    
  {}     
\theoremstyle{theorem}
\newtheorem{theorem}{Theorem}
\newtheorem{corollary}[theorem]{Corollary}
\newtheorem{remark}[theorem]{Remark}
\newtheorem{proposition}[theorem]{Proposition}
\newtheorem{lemma}[theorem]{Lemma}
\newtheorem{conjecture}{Conjecture}

\newtheorem{definition}[theorem]{Definition}

\linethickness{0.5pt}

\newcommand{\ic}{\ensuremath{\mathcal{I}}}

\newcommand{\oc}{\ensuremath{\mathcal{O}}}

\newcommand{\fc}{\ensuremath{\mathcal{F}}}

\newcommand{\ec}{\ensuremath{\mathcal{E}}}

\newcommand{\cc}{\ensuremath{\mathcal{C}}}
\newcommand{\lc}{\ensuremath{\mathcal{L}}}

\newcommand{\tc}{\ensuremath{\mathcal{T}}}

\newcommand{\Pt}{\mathbb{P}^3}

\newcommand{\Ptw}{\mathbb{P}^2}
\newcommand{\Pq}{\mathbb{P}^4}

\newcommand{\Pn}{\mathbb{P}^n}

\newcommand{\bZ}{\mathbb{Z}}

\newcommand{\bC}{\mathbb{C}}

\newcommand{\bP}{\mathbb{P}}

\newcommand{\lG}{\lambda}

\newcommand{\oL}{\overline}

\newcommand{\dual}{\spcheck}

\newcommand{\bds}{\begin{displaystyle}}
\newcommand{\eds}{\end{displaystyle}}

\title[Spaces of matrices of constant rank.]{Spaces of matrices of constant rank\\ and uniform vector bundles.}

\author{Ph. ELLIA- P. MENEGATTI}
\address{Dipartimento di Matematica e Informatica, 35 via Machiavelli, 44100 Ferrara}
\email{Philippe Ellia: phe@unife.it\newline
\indent Paolo Menegatti: paolo.menegatti@student.unife.it}


\date{\today}

\begin{document}
\maketitle


\begin{abstract} We consider the problem of determining $l(r,a)$, the maximal dimension of a subspace of $a\times a$ matrices of rank $r$. We first review, in the language of vector bundles, the known results. Then using known facts on uniform bundles we prove some new results and make a conjecture. Finally we determine $l(r;a)$ for every $r$, $1 \leq r \leq a$, when $a \leq 10$, showing that our conjecture holds true in this range.
\end{abstract}

\thispagestyle{empty}

\section*{Introduction.}
Let $A, B$ be $k$-vector spaces of dimensions $a,b$ ($k$ algebraically closed, of characteristic zero). A sub-vector space $M \subset \lc (A,B)$ is said to be of (constant) rank $r$ if every $f \in M, f \neq 0$, has rank $r$. The question considered in this paper is to determine $l(r,a,b) := max\,\{\dim M \mid M\subset \lc (A,B)$ has rank $r\}$.  This problem has been studied some time ago by various authors (\cite{Westwick72}, \cite{Sylvester}, \cite{Beasley},\cite{Eisenbud-Harris-low rk}) and has been recently reconsidered, especially in its (skew) symmetric version (\cite{Ilic-Landsberg}, \cite{Manivel-Mezzetti}, \cite{Fania-Mezzetti}, \cite{Boralevi-Faenzi-Mezzetti}).

This paper is organized as follows. In the first section we recall some basic facts. It is known, at least since \cite{Sylvester}, that to give a subspace $M$ of constant rank $r$, dimension $n+1$, is equivalent to give an exact sequence: $0 \to F \to a.\oc (-1) \stackrel{\psi}{\to} b.\oc \to E \to 0$, on $\Pn$, where $F, E$ are vector bundles of ranks $(a-r)$, $(b-r)$. We observe that the bundle $\ec := Im(\psi )$, of rank $r$, is \emph{uniform}, of splitting type $(-1^c,0^{r-c})$, where $c := c_1(E)$ (Lemma \ref{L-ec uniform}).

Then in Section two, we set $a=b$ to fix the ideas and we survey the known results (at least those we are aware of), giving a quick, uniform (!) treatment in the language of vector bundles. In Section three, using known results on uniform bundles, we obtain a new bound on $l(r;a)$ in the range $(2a+2)/3 > r > (a+2)/2$ (as well as some other results, see Theorem \ref{T-le thm}). By the way we don't expect this bound to be sharp. Indeed by ''translating'' (see Proposition \ref{P-conjecturale}) a long standing conjecture on uniform bundles (Conjecture \ref{Cj-homo}), we conjecture that $l(r;a) = a-r+1$ in this range (see Conjecture \ref{Cj-l(r;a)}). Finally, with some ad hoc arguments, we show in the last section, that our conjecture holds true for $a \leq 10$ (actually we determine $l(r,a)$ for every $r$, $1 \leq r \leq a$, when $a \leq 10$).

\section{Generalities.}

Following \cite{Sylvester}, to give $M \subset Hom(A,B)$, a sub-space of constant rank $r$, with $\dim (M)=n+1$, is equivalent to give on $\Pn$, an exact sequence:
\begin{equation}
\label{eq:exact seq}
\xymatrix{0\ar [r]& F_M\ar [r]&a.\oc (-1)\ar [rr]^{\psi _M} \ar@{->>}[rd]& &b.\oc \ar [r]&E_M\ar[r]&0 \\
 & & &\ec _M \ar@{^{(}->}[ru]& & & }
\end{equation}

\noindent where $\ec _M= Im(\psi _M), F_M, E_M$ are vector bundles of ranks $r, a-r, b-r$ (in the sequel we will drop the index $M$ if no confusion can arise).

Indeed the inclusion $i: M \hookrightarrow Hom(A,B)$ is an element of $Hom(M, A\dual \otimes B) \simeq M\dual \otimes A\dual \otimes B$ and can be seen as a morphism $\psi : A\otimes \oc \to B\otimes \oc (1)$ on $\bP (M)$ (here $\bP (M)$ is the projective space of lines of $M$). At every point of $\bP (M)$, $\psi$ has rank $r$, so the image, the kernel and the cokernel of $\psi$ are vector bundles. 

A different (but equivalent) description goes as follows: we can define $\psi : A\otimes \oc (-1) \to B\otimes \oc$ on $\bP (M)$, by $(v, \lG f) \to \lG f(v)$.

The vector bundle $\ec _M$ is of a particular type.
\begin{definition}
A rank $r$ vector bundle, $E$, on $\Pn$ is \emph{uniform} if there exists $(a_1,...,a_r)$ such that $E_L \simeq \bigoplus _{i=1}^r \oc _L(a_i)$, for every line $L \subset \Pn$ ($(a_1,...,a_r)$ is the splitting type of $E$, it is independent of $L$).

The vector bundle $F$ is \emph{homogeneous} if $g^*(F)\simeq F$, for every automorphism of $\Pn$.
\end{definition}

Clearly a homogeneous bundle is uniform (but the converse is not true).   

The first remark is:

\begin{lemma}
\label{L-ec uniform}
With notations as in (\ref{eq:exact seq}), $c_1(E_M)\geq 0$ and $\ec _M$ is a uniform bundle of splitting type $(-1^c,0^b)$, where $c=c_1(E_M), b = r-c$.
\end{lemma}

\begin{proof} Since $E$ is globally generated, $c_1(E)\geq 0$ (look at $E_L$). Let $\ec _L=\bigoplus \oc _L(a_i)$. We have $a_i \geq -1$, because $a.\oc _L(-1) \twoheadrightarrow \ec_L$. We have $a_i \leq 0$, because $\ec _L \hookrightarrow b.\oc _L$. So $-1 \leq a_i \leq 0, \forall i$. Since $c_1(\ec )=-c_1(E)$ the splitting type is as asserted and does not depend on the line $L$.
\end{proof}

The classification of rank $r \leq n+1$ uniform bundles on $\Pn$, $n\geq 2$, is known (\cite{VdV}, \cite{Elencwajg}, \cite{Elencwajg-Hirschowitz-Schneider}, \cite{Ballico-n+1}):

\begin{theorem}
\label{T-unif r<=n+1}
A rank $r \leq n+1$ uniform vector bundle on $\Pn$, $n \geq 2$, is one of the following: $\bigoplus ^r \oc (a_i)$, $T(a)\oplus k.\oc (b)$, $\Omega (a)\oplus k.\oc (b)$ ($0 \leq k \leq 1$), $S^2T_{\Ptw}(a)$.
\end{theorem}

We will use the following result (see \cite{Ein}):

\begin{theorem}{\emph{(Evans-Griffith)}}\\
\label{Evans-Griffith}
Let $\fc$ be a rank $r$ vector bundle on $\Pn$, then $\fc$ is a direct sum of line bundles if and only if $H^i_*(\fc )=0$, for $1 \leq i \leq r-1$.
\end{theorem}

The first part of the following Proposition is well known, the second maybe less.

\begin{proposition}
\label{P-surj psi}
Assume $n \geq 1$.\\
(1) If $a \geq b+n$ the generic morphism $a.\oc _{\Pn} \to  b.\oc _{\Pn}(1)$ is surjective.\\
(2) If $a < b+n$ no morphism $a.\oc _{\Pn} \to  b.\oc _{\Pn}(1)$ can be surjective.
\end{proposition}

\begin{proof} (1) It is enough to treat the case $a=b+n$ and, by semi-continuity, to produce one example of surjective morphism. Consider $$\Psi = \left( \begin{array}{ccccccc}
x_0 & \cdots & x_n & 0 & \cdots & & 0 \\
0 & x_0& \cdots & x_n & 0 &\cdots & 0\\
\vdots & \vdots & & & & & \\
0 & \cdots &  0 &x_0 & \cdots & &  x_n\end{array}\right)$$
(each row contains $b-1$ zeroes). It is clear that this matrix has rank $b$ at any point. For a more conceptual (and complicated) proof see \cite{VO1}, Prop. 1.1.\\  
(2) If $n=1$, the statement is clear. Assume $n \geq 2$. If $\psi$ is surjective we have $0 \to K \to a.\oc \to b.\oc (1) \to 0$ and $K$ is a vector bundle of rank $r=a-b < n$. Clearly we have $H^i_*(K\dual )=0$ for $1\leq i \leq r-1 \leq n-2$. By Evans-Griffith's theorem, $K$ splits as a direct sum of line bundles, hence the exact sequence splits ($n \geq 2$) and this is absurd. 

This can also be proved by a Chern class computation (see \cite{Sylvester}).
\end{proof}

From now on we will assume $A=B$ and write $l(r;a)$ instead of $l(r;a,a)$.

\section{Known results.}

We begin with some general facts:

\begin{lemma}
\label{L-ec splits}
Assume the bundle $\ec$ corresponding to $M \subset End(A)$ of constant rank $r$, $\dim (A)=a$, is a direct sum of line bundles. Then $\dim (M) \leq a-r+1$.
\end{lemma}

\begin{proof} Let $\dim (M)=n+1$ and assume $\ec = k.\oc (-1)\oplus (r-k).\oc$. If $k=0$, the surjection $a.\oc (-1) \twoheadrightarrow \ec \simeq r.\oc$, shows that $a \geq r+n$ (see Proposition \ref{P-surj psi}). If $k>0$, we have $0 \to k.\oc (-1) \to (a-r+k).\oc \to E\to 0$. Dualizing we get: $(a-r+k).\oc \twoheadrightarrow k.\oc (1)$, hence (always by Proposition \ref{P-surj psi}) $a-r+k \geq k+n$. So in any case $a-r \geq n$.
\end{proof}

\begin{lemma}
\label{L-l(r,a,a) geq a-r+1}
For every $r$, $1 \leq r \leq a$, we have $l(r;a) \geq a-r+1$.
\end{lemma}

\begin{proof} Set $n = a-r$. On $\Pn$ we have a surjective morphism $a.\oc (-1) \stackrel{\oL \psi}{\to} r.\oc$ (Proposition \ref{P-surj psi}). Composing with the inclusion $r.\oc \hookrightarrow r.\oc \oplus (a-r).\oc$, we get $\psi :a.\oc (-1) \to a.\oc$, of constant rank $r$.
\end{proof}

Finally we get:

\begin{proposition}
\label{P-a geq 2r}
(1) We have $l(r;a) \leq max\,\{r+1,a-r+1\}$\\
(2) If $a \geq 2r$, then $l(r;a)= a-r+1$.
\end{proposition}

\begin{proof} (1) Assume $r+1 \geq a-r+1$. If $\dim (M)=l(r,a)=n+1$ and if $r < n$, then (\cite{Elencwajg-Hirschowitz-Schneider}) $\ec$ is a direct sum of line bundles and $n \leq a-r$. But then $r < n \leq a-r$, against our assumption. So $r+1 \geq n+1=l(r;a)$.

If $a-r \geq r$. If $n > a-r$, then $n > r$ and this implies that $\ec$ is a direct sum of line bundles. Hence $n \leq a-r$.

(2) We have $max\,\{r+1,a-r+1\}= a-r+1$ if $ a \geq 2r$. So $l(r;a) \leq a-r+1$ by (1). We conclude with Lemma \ref{L-l(r,a,a) geq a-r+1}.
\end{proof}

\begin{remark} Proposition \ref{P-a geq 2r} was first proved (by a different method) by Beasley \cite{Beasley}.
\end{remark}

Very few indecomposable rank $r$ vector bundles with $r<n$ are known on $\Pn$ ($n>4$). One of these is the bundle of Tango (see \cite{OSS}, p. 84 for details). We will use it to prove:

\begin{lemma}
\label{L-Tango}
We have $l(t+1; 2t+1)=t+2$.
\end{lemma}

\begin{proof} By Proposition \ref{P-a geq 2r} we know that $l(t+1;2t+1) \leq t+2$. So it is enough to give an example. Set $n=t+1$ and assume first $n \geq 3$. If $\tc$ denotes the Tango bundle, then we have: $0 \to T(-2) \to (2n-1).\oc \to \tc \to 0$. Dualizing we get $0 \to \tc \dual (-1) \to (2n-1).\oc (-1) \to \Omega (1) \to 0$. Combining with the exact sequence: $0 \to \Omega (1) \to (n+1).\oc \oplus (n-2).\oc \to \oc (1)\oplus (n-2).\oc \to 0$, we get a morphism $(2n-1).\oc (-1) \to (2n-1).\oc$, of constant rank $n$.

If $n=2$, using the fact that $T(-2)\simeq \Omega (1)$, from Euler's sequence, we get $3.\oc (-1) \to 3.\oc$, whose image is $T(-2)$.
\end{proof}

\begin{remark} Lemma \ref{L-Tango} was first proved by Beasley (\cite{Beasley}), by a different method.
\end{remark}

Finally on the opposite side, when $r$ is big compared with $a$, we have:

\begin{proposition}\emph{(Sylvester \cite{Sylvester})}\\
\label{L-l(a-1;a)}
We have:
$$l(a-1;a) = \left\{ \begin{array}{l}
2\,\,\,\,if \,\,a\,\,is\,\,even\\
3\,\,\,\,if \,\,a\,\,is\,\,odd \end{array}\right.$$
\end{proposition}

The proof is a Chern classes computation. The next case $a=r-2$ is more involved and there are only partial results:

\begin{proposition}\emph{(Westwick \cite{Westwick-2})}\\
\label{P-Westwick}
We have $3 \leq l(a-2;a) \leq 5$. Moreover:
\begin{enumerate}
\item $l(a-2;a,a) \leq 4$ except if $a \equiv 2, 10 \pmod{12}$ where it could be $l(a-2;a,a)=5$.
\item If $a \equiv 0 \pmod{3}$, then $l(a-2;a,a)=3$.
\item If $a \equiv 1 \pmod{3}$, we have $l(2;4,4)=3$ and $l(8;10,10)=4$ (so $a$ doesn't determine $l(a-2;a,a)$)
\item If $a \equiv 2 \pmod{3}$, then $l(a-2;a,a) \geq 4$. Moreover if $a \not \equiv 2 \pmod{4}$, then $l(a-2;a,a)=4$.
\end{enumerate}
\end{proposition}

\begin{proof} We have $C(F)=C(E)(1-h)^a$. Let $C(F)=1+s_1h+s_2h^2$, $C(E)=1+t_1h+t_2h^2$. We get $s_1 = t_1-a$ (coefficient of $h$); $s_2 = a(a-1)/2-at_1+t_2$ (coefficient of $h^2$). From the coefficient of $h^3$ it follows that: $t_2 = (a-1)[3t_1-a+2]/6$. The coefficient of $h^4$ yields after some computations: $(a+1)(a-2-2t_1)=0$.
It follows that $t_1=\dfrac{(a-2)}{2}$ and $t_2 = \dfrac{(a-1)(a-2)}{12}$ if we are on $\Pn , n \geq 4$. Finally the coefficient of $h^5$ gives $(a+1)(a+2)=0$, showing that $l(a-2;a) \leq 5$.

If we are on $\Pq$, from $t_1=(a-2)/2$ we see that $a$ is even. From $t_2=(a-1)(a-2)/12$, we get $a^2+2-3a \equiv 0 \pmod{12}$. This implies $a\equiv 1,2,5,10 \pmod{12}$. Since $a$ is even we get $a\equiv 2, 10 \pmod{12}$.

If $a=3m$ and if we are on $\Pt$, then $6t_2 = (3m-1)(3t_1-3m+2) \equiv 0 \pmod{6}$, which is never satisfied. So $l(a-2,a) \leq 3$ in this case.

The other statements follow from the construction of suitable examples, see \cite{Westwick-2}.
\end{proof}

\begin{remark}
\label{R-a-2}
On $\Pq$ a rank two vector bundle with $c_1=0$ has to verify the Schwarzenberger condition $c_2(c_2+1) \equiv 0 \pmod{12}$. If $l(a-2;a)= 5$ for some $a$, then $a = 12m+2$ or $a=12m+10$. In the first case the condition yields $m \equiv 0,5,8,9 \pmod{12}$, in the second case $m \equiv 2,3,6,11 \pmod{12}$. So, as already noticed in \cite{Westwick-2}, the lowest possible value of $a$ is $a=34$. This would give an indecomposable rank two vector bundle with Chern classes $c_1=0, c_2 = 24$. Indeed if we have an exact sequence (\ref{eq:exact seq}), $E$ and $F$ cannot be both a direct sum of line bundles (because, by Theorem \ref{Evans-Griffith}, $\ec$ would also be a direct sum of line bundles, which is impossible).  
\end{remark}

Finally we have:

\begin{proposition}\emph{(Westwick \cite{Westwick})}\\
\label{P-2a-2r+1}
For every $a,r$, $l(r;a) \leq 2a - 2r+1$
\end{proposition}

As noticed in \cite{Ilic-Landsberg} (Theorem 1.4) this follows directly from a result of Lazarsfeld on ample vector bundles. We will come back later on this bound.

\section{Further results and a conjecture.}

There are examples, for every $n \geq 2$, of uniform but non homogeneous vector bundles on $\Pn$ of rank $2n$ (\cite{Drezet}. However it is a long standing conjecture that every uniform vector bundle of rank $r < 2n$ is homogeneous. Homogeneous vector bundles of rank $r < 2n$ on $\Pn$ are classified (\cite{Ballico-Ellia-homo}, so the conjecture can be formulated as follows:

\begin{conjecture}
\label{Cj-homo}
Every rank $r < 2n$ uniform vector bundle on $\Pn$ is a direct sum of bundles chosen among: $S^2T_{\Ptw}(a)$, $\wedge ^2T_{\Pq}(b)$, $T_{\Pn}(c)$, $\Omega _{\Pn}(d)$, $\oc _{\Pn}(e)$; where $a,b,...,e$ are integers.
\end{conjecture}

The conjecture holds true if $n \leq 3$ (\cite{Elencwajg}, \cite{Ballico-Ellia}).

Before to go on we point out an obvious but useful remark.

\begin{remark}
\label{R-D-1}
Clearly an exact sequence (\ref{eq:exact seq}) exists if and only if the dual sequence twisted by $\oc (-1)$ exists. So we may replace $\ec$ by $\ec \dual (-1)$. If $\ec$ has splitting type $(-1^c,0^b)$, $\ec \dual (-1)$ has splitting type $(0^c,-1^b)$.
\end{remark}

\begin{proposition}
\label{P-conjecturale}
(1) Take $r,n$ such that $n \leq r < 2n$. Assume $a-r < n$ and that every rank $r$ uniform bundle on $\Pn$ is homogeneous. Then $l(r;a) \leq n$, except if $r=n, a=2n-1$ in which case $l(n;2n-1) = n+1$.\\
(2) Assume Conjecture \ref{Cj-homo} is true. Then $l(r;a)=a-r+1$ for $r < (2a+2)/3$, except if $r =(a+1)/2$, in which case $l(r;a)=a-r+2$.
\end{proposition}

\begin{proof}
(1) In order to prove the statement it is enough to show that there exists no subspace $M$ of constant rank $r$ and dimension $n+1$ under the assumption $a-r < n$, $n \leq r < 2n$ (except if $r=n, a=2n-1$, in which case $l(n;2n-1)=n+1$ by Lemma \ref{L-Tango}).

Such a space would give an exact sequence (\ref{eq:exact seq}) with $\ec$ uniform of rank $r < 2n$ on $\Pn$. If $\ec$ is a direct sum of line bundles, by Lemma \ref{L-ec splits} we get $l(a;r)=a-r+1 <n+1$. Hence $\ec$ is not a direct sum of line bundles. Since the splitting type of $\ec$ is $(-1^c,0^{r-c})$ (Lemma \ref{L-ec uniform}), we see that: $\ec \simeq \Omega (1)\oplus k.\oc \oplus (r-k-n).\oc (-1)$, $\ec \simeq T(-2)\oplus t.\oc \oplus (r-t-n).\oc (-1)$, or, if $n=4$, $\ec \simeq (\wedge ^2 \Omega )(2)$.
\medskip

Let's first get rid of this last case. The assumption $a-r < n$ implies $a \leq 9$. It is enough to show that there is no exact sequence (\ref{eq:exact seq}) on $\Pq$, with $\ec = (\wedge ^2 \Omega )(2)$ and $a=9$. From $0 \to \ec \to 9.\oc \to E\to 0$, we get $\cc (E) = \cc (\ec )^{-1}$. From the Koszul complex we have $0 \to \ec \to \wedge ^2 V \otimes \oc \to \Omega (2) \to 0$. It follows that $\cc (E) = \cc (\Omega (2))$. Since $rk(E)=3$ and $c_4(\Omega _{\Pq}(2))=1$, we get a contradiction.
\medskip 

So we may assume $\ec \simeq \Omega (1)\oplus k.\oc \oplus (r-k-n).\oc (-1)$ or $\ec \simeq T(-2)\oplus t.\oc \oplus (r-t-n).\oc (-1)$. By dualizing the exact sequence (\ref{eq:exact seq}), we may assume $\ec \simeq \Omega (1)\oplus k.\oc \oplus (r-k-n).\oc (-1)$. The exact sequence (\ref{eq:exact seq}) yields:
$$0 \to \Omega (1)\oplus (r-n-k).\oc (-1) \to (a-k).\oc \to E \to 0\,\,\,(*)$$

Since $H^i_*(\Omega )=0$ for $2\leq i \leq n-1$, from the exact sequence $(*)$ we get $H^i_*(E)=0$, for $1 \leq i \leq n-2$. Since $rk(E)=a-r < n$, it follows from Evans-Griffith's theorem that $E \simeq \bigoplus \oc (a_i)$. We have $a_i \geq 0$, $\forall i$, because $E$ is globally generated. Moreover one $a_i$ at least must be equal to 1 (otherwise $h^1(E\dual \otimes \ec )=0$ and the sequence $(*)$ splits, which is impossible). So $a_1=1$, $a_i\geq 0, i > 1$. It follows that $h^0(E) \geq (n+1)+(a-r-1)=n+a-r$. On the other hand $h^0(E)=a-k$ from $(*)$. 

If $k < r-n$, we see that one of the $a_i$'s, $i > 1$, must be $>0$. This implies $h^0(E) \geq 2(n+1)+(a-r-2) = 2n+a-r$. So $a-k = h^0(E) \geq 2n+a-r$. Since $a \geq a-k$, it follows that $a \geq 2n+a-r$ and so $r \geq 2n$, against our assumption.

We conclude that $k=r-n$ and $E = \oc (1)\oplus (a-r-1).\oc$. In particular $\ec = \Omega (1)\oplus (r-n).\oc$ ($(*)$ is Euler's sequence plus some isomorphisms). We turn now to the other exact sequence:
$$0 \to F \to a.\oc (-1) \to \Omega (1)\oplus (r-n).\oc \to 0\,\,\,(+)$$
We have $\cc (F) = (1-h)^a.\cc (\Omega (1))^{-1}$. Here $\cc (F) = 1+c_1h+...+c_nh^n$ is the Chern polynomial of $F$ (computations are made in $\bZ [h]/(h^{n+1}$). From the Euler sequence $\cc (\Omega (1))^{-1}=1+h$. It follows that:
$$\cc (F) = (1+h).\left( \sum _{i=0}^a \binom{a}{i}(-1)^ih^i\right)$$
Since $rk(F) = a-r < n$, $c_n(F)=0$. Since $a \geq r \geq n$, it follows that $\dbinom{a}{n}=\dbinom{a}{n-1}$. This implies $a=2n-1$.   

Observe that $r \geq n$ (because $k=r-n \geq 0$). If $r \geq n+1$, then $rk(F) \leq n-2$, hence $c_{n-1}(F)=0$. This implies: $\dbinom{2n-1}{n-1}=\dbinom{2n-1}{n-2}$, which is impossible.

We conclude that $r=n$ and $a=2n-1$, so we are looking at $l(n;2n-1)$. By Lemma \ref{L-Tango} we know that $l(n;2n-1)= n+1$.
\medskip

This proves (1).
\medskip

(2) Now we apply (1) by setting $n:= a-r+1$. Clearly $n > a-r$. The condition $n \leq r < 2n$ translates in: $(a+1)/2 \leq r < (2a+2)/3$. So, under these assumptions, we get $l(r;a) \leq n=a-r+1$, except if $r=n$, $a=2n-1$. In this latter case we know that $l(n;2n-1)=n+1$ (Lemma \ref{L-Tango}). We conclude with Lemma \ref{L-l(r,a,a) geq a-r+1}. 
\end{proof}

Since Conjecture \ref{Cj-homo} is true for $r \leq n+1$ and $n=3, r=5$ (\cite{Ballico-Ellia}), we may summarize our results as follows:

\begin{theorem}\quad \\
\label{T-le thm}
\noindent(1) If $r \leq a/2$, then $l(r,a) = a-r+1$\\
(2) If $a$ is odd, $l(\frac{a+1}{2};a) = \frac{a+1}{2}+1$ $(=a-r+2$)\\
(3) If $ \frac{(2a+2)}{3} > r \geq \frac{a}{2}+1$, then $l(r;a) \leq r-1$.\\
(4) If $a$ is even: $l(\frac{a}{2}+1;a)= \frac{a}{2}$ $(=a-r+1)$.\\
(5) If $r \geq (2a+2)/3$, then $l(r,a) \leq 2(a-r)+1$\\
(6) We have $l(5;7)=3$ ($=a-r+1$). 
\end{theorem}

\begin{proof} (1) This is Proposition \ref{P-a geq 2r}.\\
(2) This is Lemma \ref{L-Tango}.\\
(3) Set $n=r-1$. Uniform vector bundles of rank $r=n+1$ on $\Pn$ are homogeneous. We have $n \leq r < 2n$ if $r \geq 3$ and $a-r < n$ if $r \geq (a/2)+1$. If $r \leq 2$ and $r \geq (a/2)+1$, then $a \leq 2$. Hence $r=a=2$ and $l(2;2)=1$. So the assumption of Proposition \ref{P-conjecturale}, (1) are fulfilled. We conclude that $l(r,a) \leq r-1$.\\
(4) Follows from (3) and Lemma \ref{L-l(r,a,a) geq a-r+1}.\\
(5) This is Proposition \ref{P-2a-2r+1}.\\
(6) Since uniform vector bundles of rank 5 on $\Pt$ are homogeneous, this follows from Proposition \ref{P-conjecturale} (1) and Lemma \ref{L-l(r,a,a) geq a-r+1}.
\end{proof}

\begin{remark} Point 3 of the theorem improves the previous bound of Beasley but we don't expect this bound to be sharp (see Conjecture \ref{Cj-l(r;a)}). Points 4 and 6 also are new. The bound of (5) is so far the best known bound in this range. It is reached for some values of $a$ in the case $r=a-1$ (Proposition \ref{L-l(a-1;a)}), but already in the case $r=a-2$ we don't know if it is sharp.
\end{remark}

It is natural at this point to make the following:

\begin{conjecture}
\label{Cj-l(r;a)} Let $a,r$ be integers such that $(2a+2)/3 > r > (a/2)+1$, then $l(r;a) = a-r+1$.
\end{conjecture}

\begin{remark}
This conjecture should be easier to prove than Conjecture \ref{Cj-homo}, indeed in terms of vector bundles it translates as follows: every rank $r < 2n$ uniform vector bundle, $\ec$, fitting in an exact sequence (\ref{eq:exact seq}) on $\Pn$ is homogeneous.
\medskip

By the way the condition $r < 2n$ seems necessary. If $n=2$ this can be seen as follows. Consider the following matrix (taken from \cite{Sylvester}):
$$\Psi = \left( \begin{array}{ccccc}
0 & -x_2 & 0 & -x_0 & 0\\
x_2 & 0 & 0 & -x_1 & -x_0\\
0 & 0 & 0& -x_2 & -x_1\\
x_0 & x_1 & x_2 & 0 & 0\\
0 & x_0 & x_1 & 0 & 0\end{array}\right)$$
It is easy to see that $\Psi$ has rank four at any point of $\Ptw$, hence we get:
$$0 \to \oc (b) \to 5.\oc (-1) \stackrel{\Psi}{\to}5.\oc \to \oc (c) \to 0$$
with $\ec = Im(\Psi )$ a rank four uniform bundle. On the line $L$ of equation $x_2=0$, $\Psi$ can be written:
$$\begin{array}{ccc}
\oc _L(-3)\hookrightarrow & 3.\oc _L(-1)\twoheadrightarrow 2.\oc _L & \\
 & \oplus & \\
 &2.\oc _L(-1) \hookrightarrow 3.\oc _L & \twoheadrightarrow \oc _L(2)\end{array}$$
 It follows that $b=-3, c=2$ and the splitting type of $\ec$ is $(-1^2, 0^2)$. Now rank four homogeneous bundles on $\Ptw$ are classified (Prop. 3, p.18 of \cite{Drezet-Crelle}) and are direct sum of bundles chosen among $\oc (a), T(b), S^2T(c), S^3T(d)$. If $\ec$ is homogeneous the only possibility is $\ec (1) \simeq T(-1)\oplus \oc (1)\oplus \oc$, but in this case the exact sequence $0 \to \oc (-2) \to 5.\oc \to \ec (1) \to 0$, would split, which is absurd. We conclude that $\ec$ is not homogeneous. In fact $\ec$ is one of the bundles found by Elencwajg (\cite{Elencwajg-2}).
\end{remark}

\begin{remark}
\label{R-a<=8} The results of this section and the previous one determine $l(r;a)$ for $a \leq 8$, $1 \leq r \leq a$. To get a complete list for $a \leq 10$, we have to show, according to Conjecture \ref{Cj-l(r;a)}, that $l(6;9) = l(7;10) = 4$. This will be done in the next section.
\end{remark}

\section{Some partial results.}

In the following lemma we relax the assumption $r < 2n$ in Proposition \ref{P-conjecturale} when $c_1(\ec (1))=1$.

\begin{lemma}
\label{L-c1ec(1)=1}
Assume we have an exact sequence (\ref{eq:exact seq}) on $\Pn$, with $rk(F)=a-r < n$  and $c_1(\ec (1))=1$. Then $a = 2n-1$ and $r=n$.
\end{lemma}

\begin{proof} If $c_1(\ec (1))=1$, $\ec (1)$ has splitting type $(1, 0^{r-1})$. It follows from \cite{Ellia-Thèse}, Prop. IV, 2.2, that $\ec (1) = \oc (1)\oplus (r-1).\oc$ or $\ec (1) = T(-1)\oplus (r-n).\oc$. From the exact sequence $0 \to F(1) \to a.\oc \to \ec(1) \to 0$, we get $c_1(F(1))=-1$. Since $F(1) \hookrightarrow a.\oc$, it follows that $F(1)$ is uniform of splitting type $(-1,0^{a-r-1})$. Since $rk(F) < n$, $F(1) = \oc (-1)\oplus (a-r-1).\oc$. This shows that necessarily $\ec (1) = T(-1)\oplus (r-n).\oc$. Now from the exact sequence: $0 \to T(-1)\oplus (r-n).\oc \to a.\oc (1) \to E(1) \to 0$, we get $\cc (E(1)) = (\cc (t(-1))^{-1}(1+h)^a$, i.e. $\cc (E(1))= (1-h)(1+h)^a$. Since $rk(E) < n$, we have $c_n(E(1))=0$ and arguing as in the proof of Proposition \ref{P-conjecturale}, we get $a=2n-1, r=n$.
\end{proof}

\begin{remark}
\label{R-c1ec(1)=1}
Since we know that $l(n;2n-1)= n+1$ (Lemma \ref{L-Tango}), we may, from now on, assume $c_1(\ec (1)) \geq 2$.
\end{remark}

Since $\ec (1)$ is globally generated, taking $r-1$ general sections we get:
\begin{equation}
\label{eq:r-1}
0 \to (r-1).\oc \to \ec (1) \to \ic _X(b) \to 0
\end{equation}

Here $X$ is a pure codimension two subscheme, which is smooth if $n \leq 5$ and which is irreducible, reduced, with singular locus of codimension $\geq 6$, if $n \geq 6$.

\begin{lemma}
\label{L-X aCM}
Assume $n \geq 3$ and $rk(F) < n$. If $X$ is arithmetically Cohen-Macaulay (aCM), i.e. if $H^i_*(\ic _X) =0$ for $1 \leq i \leq n-2$, then $F$ is a direct sum of line bundles.
\end{lemma}

\begin{proof} From (\ref{eq:r-1}) we get $H^i_*(\ec )=0$ for $1 \leq i \leq n-2$. By Serre duality $H^i_*(\ec \dual )=0$, for $2\leq i \leq n-1$. From the exact sequence $0 \to \ec \dual \to a.\oc (1) \to F\dual \to 0$, we get $H^i_*(F\dual )=0$, for $1 \leq i \leq n-2$. Since $F\dual$ has rank $<n$, by Evans-Griffith theorem we conclude that $F\dual$ (hence also $F$) is a direct sum of line bundles.
\end{proof}

\begin{proposition}
\label{P-b>3}
Assume that we have an exact sequence (\ref{eq:exact seq}) on $\Pq$ with $rk(F) < 4$. Let $(-1^c,0^{r-c})$ be the splitting type of $\ec$. If $r > 4$ and if $F$ is not a direct sum of line bundles, then $c,r-c \geq 4$; in particular $rk(\ec ) \geq 8$.
\end{proposition}

\begin{proof} Assume $c$ or $b:=r-c$ $< 4$. By dualizing the exact sequence \ref{eq:exact seq} if necessary, we may assume $b < 4$. We have an exact sequence (\ref{eq:r-1}):
$$ 0 \to (r-1).\oc \to \ec (1) \to \ic _X(b) \to 0$$
where $X \subset \Pq$ is a smooth surface of degree $d=c_2(\ec (1))$. If $b <3$, $X$ is either a complete intersection $(1,d)$ or lies on a hyper-quadric. In any case $X$ is a.C.M. By Lemma \ref{L-X aCM}, $F$ is a direct sum of line bundles. 

Assume $b=3$. From the classification of smooth surfaces in $\Pq$ we know that if $d \leq 3$, then $X$ is a.C.M. Now $X$ is either a complete intersection $(3,3)$, hence a.C.M. or linked to a smooth surface, $S$, of degree $9-d$ by such a complete intersection. If $S$ is a.C.M. the same holds for $X$. From the classification of smooth surfaces of low degree in $\Pq$, if $X$ is not a.C.M. we have two possibilities:\\
(i) $X$ is a Veronese surface and $S$ is an elliptic quintic scroll,\\
(ii) $X$ is an elliptic quintic scroll and $S$ is a Veronese surface.
\medskip

(i) If $X=V$ is a Veronese surface then we have an exact sequence:
$$0 \to 3.\oc  \to \Omega (2) \to \ic _V(3) \to 0$$
It follows that $\cc (\ec (1)) = \cc (\Omega (2)) = 1 + 3h + 4h^2 + 2h^3 + h^4$. So $\cc (F(1)) = (\cc (\ec (1))^{-1} = 1-3h+5h^2-5h^3$. It follows that $F$ (and hence $E$ also) has rank three. From $\cc (E(1)) = (1+h)^a.\cc(F(1))$ and $c_4(E(1))=0$, we get $0 = a(a-5)(a-6)(a-7)$. So $a \leq 7$. Since $a = rk(E) + r$, we get a contradiction.
\medskip

\noindent (ii) If $X = E$ is an elliptic quintic scroll, then we have:
$$0 \to T(-2) \to 5.\oc \to \ic _E(3) \to 0$$
It follows that $\cc (\ec (1)) = \cc (T(-2))^{-1}$ and $\cc (F(1)) = \cc (T(-2)) = 1-3h+4h^2-2h^3+h^4$, in contradiction with $rk(F) < 4$.
\end{proof}

\begin{lemma}
Assume we have an exact sequence (\ref{eq:exact seq}) on $\Pq$ with $a-r < 4$. If $r > 4$ and if $F$ is a direct sum of line bundles, then $rk(\ec ) \geq 8$.
\end{lemma}

\begin{proof} If $r=5$ we conclude with Theorem \ref{T-le thm}, (3), (6). If $r=6$, then $a \leq 9$ and it is enough to show that $l(6;9) \leq 4$ i.e. that there is no exact sequence (\ref{eq:exact seq}) on $\Pq$. In the same way, if $r=7$ it is enough to show that $l(7;10) \leq 4$.

If $r=6$, we may assume that the splitting type of $\ec$ is $(-1^1,0^5), (-1^2, 0^4), (-1^3,0^3)$. By dualizing and by Lemma \ref{L-c1ec(1)=1} we may disregard the first case. It follows that $c_1(F)=-7$ or $-6$. If $r=7$, in a similar way, we may assume that the splitting type of $\ec$ is $(-1^2,0^5)$ or $(-1^3,0^4)$. So $c_1(F)=-7$ or $-8$.

Let $\cc (F(1)) = (1-f_1h)(1-f_2h)(1-f_3h)$. We have $\cc (E(1)) = (1+h)^a\cc (F(1))$. From $c_4(E(1))=0$ we get:
$$\psi (a):= a^3 -a^2(4s+6) + a (12d+12s+11)-12d-8s-24t-6 =0$$
Where $s = -c_1(F(1))= \sum f_i$, $d =c_2(F(1)) = \sum _{i<j}f_if_j$, $t =-c_3(F(1)) = \prod f_i$. We have $f_i\geq 0, \forall i$ and $3 \leq s \leq 5$.

We have to check that this equality can't be satisfied for $a = 9, 10$. We have $\psi (9) = 8(42-28s + 12d -37)$. If $\psi (9)=0$ we get $3\mid s$. It follows that $s=3$. So the condition is: $4d-t = 14$. If one of the $f_i$'s is zero, then $t=0$ and we get a contradiction. So $f_i > 0, \forall i$ and the only possibility is $(f_i) = (1,1,1)$, but then $4d-t = 11 \neq 14$. 

For $a=10$, we get $\psi (10) = 504 -288s + 108d -24t$. If $f_1=f_2 =0$, then $d=t=0$ and we get $s= 504/288$ which is not an integer. If $f_1=0$, then $t=0$, $d=f_2f_3$, $s=f_2+f_3$. If $s \geq 4$, $504 + 108d = 288s \geq 1152$. It follows that $d \geq 6$. If $d=6$ we have necessarily $s=5$ and $\psi (10) \neq 0$. So $s=3$ and $d=2$, but also in this case $\psi (10)\neq 0$. We conclude that $f_i>0, \forall i$. So we are left with $(f_i) = (1,1,1)$, $(1,1,2)$, $(1,1,3)$, $(1,2,2)$. In any of these cases one easily checks that $\psi (10) \neq 0$.
\end{proof}

\begin{corollary}
We have $l(6;9)= l(7;10) =4$. In particular $l(r,a)$ is known for $a \leq 10$ and $1 \leq r \leq a$ and Conjecture \ref{Cj-l(r;a)} holds true for $a \leq 10$.
\end{corollary}

\begin{proof} We have seen that $l(6;9), l(7;10) \leq 4$, by Lemma \ref{L-l(r,a,a) geq a-r+1} we have equality. Then all the other values of $l(r;a)$ are given by Theorem \ref{T-le thm}, Proposition \ref{L-l(a-1;a)} and Proposition \ref{P-Westwick}, if $a \leq 10$.
\end{proof}


\end{document}